%% file: main.tex
\documentclass[11pt,a4paper,reqno]{amsart}
\input{packages}

\input{defs}
\input{theoremEnv}

\keywords{random tree; reinforcement; neural network; small-world graph}
\subjclass[2010]{60K35; 82C22}

\begin{document}
\author{Markus Heydenreich}
\author{Christian Hirsch}
\address[Markus Heydenreich]{Mathematisches Institut, Ludwig-Maximilians-Universit\"at M\"unchen, Theresienstra\ss e 39, 80333 M\"unchen, Germany}
\email{m.heydenreich@lmu.de}
\address[Christian Hirsch]{Institut f\"ur Mathematik, Universit\"at Mannheim, B6, 26, 68161 Mannheim, Germany.}
\email{hirsch@uni-mannheim.de}

\thanks{This work is supported by The Danish Council for Independent Research | Natural Sciences, grant DFF -- 7014-00074 \emph{Statistics for point processes in space and beyond}, and by the \emph{Centre for Stochastic Geometry and Advanced Bioimaging}, funded by grant 8721 from the Villum Foundation.}

\title[A spatial small-world graph from activity-based reinforcement]{A spatial small-world graph arising from activity-based reinforcement}

\date{\today}

\begin{abstract}
	        In the classical preferential attachment model, links form instantly to newly arriving nodes and do not change over time. We propose a hierarchical random graph model in a spatial setting, where such a time-variability arises from an activity-based reinforcement mechanism. We show that the reinforcement mechanism converges, and prove rigorously that the resulting random graph exhibits the small-world property. A further motivation for this random graph stems from modeling synaptic plasticity.
\end{abstract}

\maketitle

	\input{intro}
\input{model}
\input{proofs}

\bibliography{lit}
\bibliographystyle{abbrv}

\end{document}

%% file: packages.tex
\usepackage{tikz}
\usepackage{amsmath,bm,bbm} 
\usepackage{amssymb}
\usepackage{fullpage}
\usepackage{color}
\usepackage{animate}
\usepackage{etoolbox}
\usepackage{ulem}
\usepackage{comment}
\usepackage{pgf}
\usetikzlibrary{calc}
\usetikzlibrary{patterns}
\usetikzlibrary{arrows}
\usetikzlibrary{decorations.pathreplacing}
\usepackage[utf8]{inputenc}
\usepackage{pgfplots}
\usepackage{hyperref}

%% file: defs.tex
\definecolor{wiasblue}   {cmyk}{1.0, 0.60, 0, 0}

\def\Z{\mathbb{Z}}
\def\E{\mathbb{E}}
\def\P{\mathbb{P}}

\def\mc{\mathcal}
\def\ms{\mathsf}

\def\one{\mathbbm{1}}
\def\e{\varepsilon}
\def\t{\tau}
\def\fmv{v^{\ms{max}}}
\def\aa{a}
\def\a{\beta}
\def\g{\gamma}
\def\EE{\mc E }
\def\FF{\mc F }

\def\mx{\ms{max}}
\def\NN{\mc N}

%% file: theoremEnv.tex
\newtheorem{theorem}{Theorem}[section]

\newtheorem{lemma}[theorem]{Lemma}

\theoremstyle{definition}

%% file: intro.tex

\section{Introduction}
\label{introSec}

\paragraph{\bf Network formation driven by reinforcement.}
Since the introduction of the \emph{preferential attachment} model by Barab\'asi and Albert \cite{BarabasiAlbert99}, reinforcement mechanisms are recognized as a versatile tool in network formation. 
Why are preferential attachment models so popular? On the one hand, the resulting graphs exhibit universal features that are ubiquitous in real networks e.g., scale-free property, short distances \cite{BarabasiAlbert99,BollobasRiordan04}. In spatial versions of the preferential attachment mechanism, there is even strong local clustering \cite{JacobMoerters13}. 
A second reason for the popularity lies in the plausibility of the reinforcement scheme: When new agents enter the system, then they are more likely to link with highly connected agents than with those that have only few connections. The result is that a high degree is reinforced, some authors coin this the ``Matthew effect''. 
Even though the preferential attachment model is in principal a dynamical model, the formation of edges occurs instantly, and is not changed with time, except for the addition of edges from new vertices. Variability in the formation of edges is thus not included in the preferential attachment model.  

Reinforcement effects are also typical for social sciences. 
Pemantle and Skyrms \cite{SkyrmsPemantle04,SkyrmsPemantle00} study a mathematical model for a group of agents interacting with each other in such a way that every interaction makes the same interaction in the future more probable. 
Of particular interest is the long-term behavior: both on finite graphs \cite{BonacichLiggett03} and on infinite networks \cite{LiggettRolles04} a nice characterization of the
equilibrium states can be given: The reinforcement in the model is so strong that in the long run there is a formation of groups such that only the agents inside the groups interact
but not across the groups. More precisely, it is shown that in an extremal equilibrium, the 
set of agents decomposes into finite sets, each of which includes a ``center'' that is always chosen by the other agents in that set. 

\paragraph{\bf Neural networks.}
Reinforcement mechanisms are also typical for neural networks in the context of synaptic plasticity. 
To this end, we are considering a fairly simplistic model of a neural network: 
There is a set of neurons, each of them equipped with one axon and a number of dendrites which are connected to axons of other neurons. Pairs of axons and dendrites may form synapses, which are functional connections between neurons. 
However, not all geometric connections necessarily also form functional connections. The resulting network can be interpreted as a directed graph with neurons as vertices and synapses as edges (directed from dendrite to axon). 

Experimental observation shows that the resulting neural 
network is rather sparse and very well connected, that is, any pair of neurons is connected through a short chain of neural connections reminiscent of the ``small-world property''. 
These features allow for very fast and efficient signal processing. 
The challenge is to explain the mechanism behind the 
formation of such sophisticated neural networks. 
Kalisman, Silberberg, and Markram \cite{kalisman-silberberg-markram-2005} 
use experimental evidence to advocate a \emph{tabula rasa approach} to the formation 
of these networks: In an early stage, there is a (theoretical) 
all-to-all geometrical connectivity. Stimulation and transmission 
of signals enhance certain touches to ultimately form functional 
connections, which results in a network with rather few actual synapses. 
This describes the plasticity of the brain at an early stage of the development. 

\paragraph{\bf A mathematical model.} 
In order to model these effects mathematically, we consider a model of reinforced P\'olya urns with graph-based competition. 
In this P\'olya urn interpretation, the ``color'' of the balls in the P\'olya urn represents the edges in a given graph (namely, the potential connections or touches). 

This reinforcement scheme goes as follows: we start from a very large graph (e.g.\ the complete graph or a suitable grid), and initially equip all edges with weight one. The vertices are activated uniformly at random. If node $v$ is activated at time $t\ge0$, then it chooses a neighboring node $w$ proportional to 
\[ W_{t-}(v,w)^\a \qquad \a>0, \]
where $W_{t-}(v,w)$ is the weight of the edge $(v,w)$ just \emph{before} the activation, and the weight of the chosen edge increases by one. High weight of an edge thus means that this edge is chosen very often. We are interested in the subgraph formed by those edges, whose weight is increasing linearly in time (i.e., edges that are chosen a positive fraction of time). Following the neural interpretation of the previous paragraph, these are the edges forming actual synapses. 

The parameter $\a>0$ controls the strength of the reinforcement, we distinguish between \emph{weak} reinforcement when $\a<1$ and \emph{strong} reinforcement when $\a>1$. In the case of strong reinforcement, it appears that any stable equilibrium is concentrated on small ``islands'' which are not connected to each other; for rigorous results in this direction we refer to 
\cite{warm3,warm1,warm2}. On the other hand, if there is weak reinforcement, then all edges are contained in the limiting distribution, and thus no interesting subgraph is shaping in the limit \cite{warm4}. In the bordercase $\a=1$, there is linear reinforcement, where the classical  P\'olya urn (properly normalized) converges to a Dirichlet distribution. For our model of graph based interaction, the situation is more delicate, as it seems that the behavior for $\a=1$ resembles the subcritical regime \cite{warm5}. 

We summarize that these conventional approaches yield interesting results, but they are not versatile enough to support the tabula rasa approach from a mathematical point of view: either the resulting functional connections form small local islands, or the entire network is kept in the limit. 
One might argue that our interpretation of neural interactions with reinforced P\'olya urns is far too simplified. Indeed, there are more realistic mathematical models for brain activities (e.g.\ through a system of interacting Hawkes processes \cite{DelattreFournierHoffmann16}), but we do not expect that the overall picture as described above is changing in more sophisticated setup. 
Instead, we are proposing a different route, where we introduce layers of neurons with varying fitness, and this leads indeed to an interesting network structure.

\paragraph{\bf Our contribution.} 
In the present work, we are suggesting a new model for a network arising from reinforcement dynamics that are typical for the brain. 
Our model is built upon layers of spatial graphs, and the ability of neurons to form long connections. More precisely, in the base network of possible links, neurons at a higher layer have the potential of reaching further than neurons at lower layers, and a random fitness of neurons leads to a rapid coalescence of functional connections. We prove that the resulting graph is connected and loop-free, and that far-away vertices are linked through a few edges only (``small-world''). In contrast to the preferential attachment model, which is based on reinforcement of degrees, our model reinforces edge activities.

Our main interest lies in the understanding of a versatile mathematical model for neural applications. It is clear that the actual formation of the brain involves much more complex processes that are beyond the scope of a rigorous treatment. 
Yet, we aim at clarifying which network characteristics can be explained by a simple reinforcement scheme, and which cannot.

The generality of our approach has the potential to be applied in a variety of contexts with different interpretations. Indeed, networks based on layered graphs are fundamental objects in machine learning, and therefore our model could contribute towards enhancing the understanding of ``biologically plausible deep learning'' in the spirit of \cite{BengioLBL15}.

\paragraph{\bf Future work.} 
In the current model, the growing range of neurons at higher layers is defined externally. It appears desirable to extend the model such that this feature emerges from an intrinsic mechanism of self-organization. Even though in the present setting, we are deriving our results for layers of one-dimensional graphs, we expect that the main results also hold for higher-dimensional lattices. For example, when modelling the neurons in the visual pathway, layers of two-dimensional graphs seems more appropriate. Finally, it would be of interest to test the relevance of the proposed model with measurements in real world networks. 

%% file: model.tex

\section{Model and Results}

We consider a stochastic process of dynamically evolving edge weights $\{W_t(e)\}_{{t \ge 0, e \in E}}$ on the graph with nodes $V = \Z \times \Z_{\ge0}$ and edges $E$ given by pairs $((k, h), (\ell, h + 1))$ for $|\ell - k| \le \aa^h$ for some $\aa > 1$. Here, we think of $\Z \times \Z_{\ge 0}$ as an infinite number of layers, each consisting of infinitely many nodes. Additionally, the nodes feature iid heavy-tailed fitnesses $\{F_v\}_{v \in V}$ with tail index $\g < 1$. More precisely, we assume that $s^{-\g}\P(F_v > s)$ remains bounded away from 0 and $\infty$ as $s \to \infty$.

At time $t = 0$, all edge weights are constant equal to 1, i.e., $W_0(e) = 1$ for every $e \in E$. To describe the evolution of $\{W_t\}_{t \ge 0}$, we equip the nodes of $V$ with independent Poisson clocks. When at node $v = (k, h)$ the clock rings, then we choose one of the adjacent nodes $w$ in the set $\NN_v = \{(\ell, h + 1):\,|\ell - k| \le \aa^h\}$ of out-neighbors and increment the weight of the incident edge by 1. According to the modeling paradigm described in Section \ref{introSec}, we prefer to choose fitter vertices and higher edge weights. 
More precisely, the probability to select $w = (\ell, h + 1)$ is proportional to 
$$F_wW_{t-}(v, w)^\a,$$
where the parameter $\a > 1$ describes the strength of the reinforcement bias. Figure \ref{fig} illustrates the random graph model after a finite number of reinforcement steps.

\begin{figure}[!bp]
	\centering
		\input{coalesc}
		\caption{Realization of the network model with parameters $a = 3$, $\a = 3/2$ and $\g = 1/5$. Node diameters represent $\log$-fitness values. Grayscales correspond to edge weights after 20 reinforcement steps.}
		\label{fig}
\end{figure}
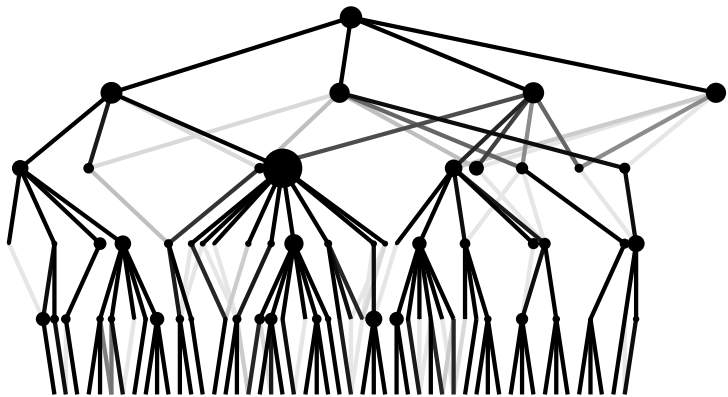

Having introduced the weight dynamics, we now extract the subgraph of relevant edges. More precisely, we let 
$$\EE = \{e \in E:\,\liminf_{t \to \infty} W_t(e) / t > 0\}$$
denote the subgraph consisting of edges that are reinforced a positive proportion of times. 

The main result of this work establishes that $\EE$ is a small-world graph in the sense that graph distances on $\EE$ between layer-0 nodes grow logarithmically in their horizontal distance. To be more precise, by translation-invariance in the first coordinate, we may fix one of the vertices to be $(0, 0)$ and therefore let $H_N$ denote the graph distance on $\EE$ between $(0, 0)$ and $(N, 0)$.

\begin{theorem}[Typical distances; multiplicative version]
	\label{thm1}
	Let $\aa, \a > 1$ and $\g < 1$. Then, asymptotically almost surely,
	$$ \frac{H_N}{\log_a(N)} \xrightarrow{N \to \infty} 2.$$
\end{theorem}

\begin{theorem}[Typical distances; additive version]
	\label{thm2}
Let $\aa, \a > 1$ and $\g < 1$. Then, there exists $c > 0$ such that for every $N, x \ge 1$
	$$\P(H_N \ge 2 \log_a(N) + x) \le \exp(-cx).$$
	In particular, $H_N$ is almost surely finite for every $N \ge 1$.
\end{theorem}

The almost sure finiteness of $H_N$ for all $N \ge 1$ means that all nodes at layer 0 are connected in $\EE$. In fact, all other nodes are connected as well. Since $\EE$ does not contain loops (Lemma~\ref{forLem}), it is therefore a tree.

%% file: coalesc.tex
\begin{tikzpicture}
\draw[line width = 1.60pt, black!10.00!white] (12.30, 0.00)--(12.45, 1.00);
\draw[line width = 1.60pt, black!10.00!white] (7.65, 0.00)--(7.65, 1.00);
\draw[line width = 1.60pt, black!10.00!white] (12.75, 0.00)--(12.60, 1.00);
\draw[line width = 1.60pt, black!10.00!white] (10.65, 0.00)--(10.80, 1.00);
\draw[line width = 1.60pt, black!10.00!white] (8.40, 0.00)--(8.55, 1.00);
\draw[line width = 1.60pt, black!10.00!white] (12.30, 0.00)--(12.30, 1.00);
\draw[line width = 1.60pt, black!10.00!white] (15.00, 0.00)--(15.00, 1.00);
\draw[line width = 1.60pt, black!10.00!white] (11.40, 0.00)--(11.55, 1.00);
\draw[line width = 1.60pt, black!10.00!white] (11.40, 0.00)--(11.40, 1.00);
\draw[line width = 1.60pt, black!15.00!white] (12.75, 0.00)--(12.90, 1.00);
\draw[line width = 1.60pt, black!15.00!white] (10.35, 0.00)--(10.20, 1.00);
\draw[line width = 1.60pt, black!20.00!white] (12.60, 0.00)--(12.75, 1.00);
\draw[line width = 1.60pt, black!20.00!white] (10.05, 0.00)--(9.90, 1.00);
\draw[line width = 1.60pt, black!50.00!white] (8.25, 0.00)--(8.10, 1.00);
\draw[line width = 1.60pt, black!55.00!white] (8.25, 0.00)--(8.25, 1.00);
\draw[line width = 1.60pt, black!85.00!white] (10.05, 0.00)--(10.20, 1.00);
\draw[line width = 1.60pt, black!85.00!white] (12.60, 0.00)--(12.45, 1.00);
\draw[line width = 1.60pt, black!85.00!white] (12.75, 0.00)--(12.75, 1.00);
\draw[line width = 1.60pt, black!90.00!white] (10.35, 0.00)--(10.35, 1.00);
\draw[line width = 1.60pt, black!90.00!white] (12.30, 0.00)--(12.15, 1.00);
\draw[line width = 1.60pt, black!90.00!white] (11.40, 0.00)--(11.25, 1.00);
\draw[line width = 1.60pt, black!95.00!white] (10.65, 0.00)--(10.50, 1.00);
\draw[line width = 1.60pt, black!95.00!white] (15.00, 0.00)--(15.15, 1.00);
\draw[line width = 1.60pt, black!95.00!white] (8.40, 0.00)--(8.25, 1.00);
\draw[line width = 1.60pt, black!95.00!white] (7.65, 0.00)--(7.50, 1.00);
\draw[line width = 1.60pt, black!100.00!white] (7.95, 0.00)--(8.10, 1.00);
\draw[line width = 1.60pt, black!100.00!white] (7.80, 0.00)--(7.65, 1.00);
\draw[line width = 1.60pt, black!100.00!white] (12.90, 0.00)--(13.05, 1.00);
\draw[line width = 1.60pt, black!100.00!white] (13.05, 0.00)--(13.20, 1.00);
\draw[line width = 1.60pt, black!100.00!white] (13.20, 0.00)--(13.20, 1.00);
\draw[line width = 1.60pt, black!100.00!white] (13.35, 0.00)--(13.20, 1.00);
\draw[line width = 1.60pt, black!100.00!white] (13.50, 0.00)--(13.65, 1.00);
\draw[line width = 1.60pt, black!100.00!white] (9.90, 0.00)--(9.90, 1.00);
\draw[line width = 1.60pt, black!100.00!white] (13.80, 0.00)--(13.65, 1.00);
\draw[line width = 1.60pt, black!100.00!white] (13.95, 0.00)--(14.10, 1.00);
\draw[line width = 1.60pt, black!100.00!white] (14.10, 0.00)--(14.10, 1.00);
\draw[line width = 1.60pt, black!100.00!white] (14.25, 0.00)--(14.10, 1.00);
\draw[line width = 1.60pt, black!100.00!white] (14.40, 0.00)--(14.55, 1.00);
\draw[line width = 1.60pt, black!100.00!white] (14.55, 0.00)--(14.55, 1.00);
\draw[line width = 1.60pt, black!100.00!white] (14.70, 0.00)--(14.55, 1.00);
\draw[line width = 1.60pt, black!100.00!white] (14.85, 0.00)--(15.00, 1.00);
\draw[line width = 1.60pt, black!100.00!white] (8.10, 0.00)--(8.10, 1.00);
\draw[line width = 1.60pt, black!100.00!white] (13.65, 0.00)--(13.65, 1.00);
\draw[line width = 1.60pt, black!100.00!white] (12.45, 0.00)--(12.45, 1.00);
\draw[line width = 1.60pt, black!100.00!white] (8.55, 0.00)--(8.70, 1.00);
\draw[line width = 1.60pt, black!100.00!white] (9.75, 0.00)--(9.90, 1.00);
\draw[line width = 1.60pt, black!100.00!white] (10.50, 0.00)--(10.35, 1.00);
\draw[line width = 1.60pt, black!100.00!white] (9.60, 0.00)--(9.75, 1.00);
\draw[line width = 1.60pt, black!100.00!white] (9.45, 0.00)--(9.30, 1.00);
\draw[line width = 1.60pt, black!100.00!white] (10.80, 0.00)--(10.95, 1.00);
\draw[line width = 1.60pt, black!100.00!white] (10.95, 0.00)--(10.95, 1.00);
\draw[line width = 1.60pt, black!100.00!white] (11.10, 0.00)--(10.95, 1.00);
\draw[line width = 1.60pt, black!100.00!white] (11.25, 0.00)--(11.10, 1.00);
\draw[line width = 1.60pt, black!100.00!white] (10.20, 0.00)--(10.35, 1.00);
\draw[line width = 1.60pt, black!100.00!white] (9.30, 0.00)--(9.15, 1.00);
\draw[line width = 1.60pt, black!100.00!white] (9.00, 0.00)--(8.85, 1.00);
\draw[line width = 1.60pt, black!100.00!white] (11.55, 0.00)--(11.70, 1.00);
\draw[line width = 1.60pt, black!100.00!white] (11.70, 0.00)--(11.70, 1.00);
\draw[line width = 1.60pt, black!100.00!white] (11.85, 0.00)--(11.70, 1.00);
\draw[line width = 1.60pt, black!100.00!white] (12.00, 0.00)--(12.00, 1.00);
\draw[line width = 1.60pt, black!100.00!white] (12.15, 0.00)--(12.00, 1.00);
\draw[line width = 1.60pt, black!100.00!white] (8.85, 0.00)--(8.85, 1.00);
\draw[line width = 1.60pt, black!100.00!white] (8.70, 0.00)--(8.85, 1.00);
\draw[line width = 1.60pt, black!100.00!white] (9.15, 0.00)--(9.15, 1.00);
\draw[line width = 1.60pt, black!100.00!white] (7.50, 0.00)--(7.35, 1.00);
\draw[line width = 1.60pt, black!10.00!white] (7.35, 1.00)--(6.90, 2.00);
\draw[line width = 1.60pt, black!10.00!white] (11.40, 1.00)--(11.85, 2.00);
\draw[line width = 1.60pt, black!10.00!white] (11.70, 1.00)--(12.00, 2.00);
\draw[line width = 1.60pt, black!10.00!white] (9.90, 1.00)--(9.60, 2.00);
\draw[line width = 1.60pt, black!10.00!white] (9.90, 1.00)--(9.45, 2.00);
\draw[line width = 1.60pt, black!10.00!white] (13.65, 1.00)--(13.80, 2.00);
\draw[line width = 1.60pt, black!10.00!white] (11.55, 1.00)--(11.70, 2.00);
\draw[line width = 1.60pt, black!10.00!white] (9.15, 1.00)--(9.45, 2.00);
\draw[line width = 1.60pt, black!10.00!white] (9.15, 1.00)--(9.30, 2.00);
\draw[line width = 1.60pt, black!15.00!white] (11.55, 1.00)--(11.85, 2.00);
\draw[line width = 1.60pt, black!20.00!white] (9.75, 1.00)--(10.05, 2.00);
\draw[line width = 1.60pt, black!85.00!white] (11.55, 1.00)--(11.10, 2.00);
\draw[line width = 1.60pt, black!85.00!white] (9.75, 1.00)--(9.30, 2.00);
\draw[line width = 1.60pt, black!90.00!white] (9.90, 1.00)--(10.35, 2.00);
\draw[line width = 1.60pt, black!90.00!white] (9.15, 1.00)--(9.00, 2.00);
\draw[line width = 1.60pt, black!95.00!white] (7.35, 1.00)--(7.50, 2.00);
\draw[line width = 1.60pt, black!95.00!white] (11.40, 1.00)--(11.10, 2.00);
\draw[line width = 1.60pt, black!95.00!white] (13.65, 1.00)--(13.95, 2.00);
\draw[line width = 1.60pt, black!95.00!white] (11.70, 1.00)--(11.70, 2.00);
\draw[line width = 1.60pt, black!100.00!white] (14.55, 1.00)--(15.00, 2.00);
\draw[line width = 1.60pt, black!100.00!white] (12.00, 1.00)--(12.30, 2.00);
\draw[line width = 1.60pt, black!100.00!white] (12.90, 1.00)--(12.90, 2.00);
\draw[line width = 1.60pt, black!100.00!white] (12.30, 1.00)--(12.30, 2.00);
\draw[line width = 1.60pt, black!100.00!white] (12.60, 1.00)--(12.30, 2.00);
\draw[line width = 1.60pt, black!100.00!white] (14.10, 1.00)--(13.95, 2.00);
\draw[line width = 1.60pt, black!100.00!white] (13.20, 1.00)--(12.90, 2.00);
\draw[line width = 1.60pt, black!100.00!white] (12.75, 1.00)--(12.30, 2.00);
\draw[line width = 1.60pt, black!100.00!white] (13.05, 1.00)--(12.90, 2.00);
\draw[line width = 1.60pt, black!100.00!white] (12.15, 1.00)--(12.30, 2.00);
\draw[line width = 1.60pt, black!100.00!white] (12.45, 1.00)--(12.30, 2.00);
\draw[line width = 1.60pt, black!100.00!white] (11.10, 1.00)--(10.65, 2.00);
\draw[line width = 1.60pt, black!100.00!white] (15.00, 1.00)--(15.15, 2.00);
\draw[line width = 1.60pt, black!100.00!white] (10.95, 1.00)--(10.65, 2.00);
\draw[line width = 1.60pt, black!100.00!white] (10.80, 1.00)--(10.65, 2.00);
\draw[line width = 1.60pt, black!100.00!white] (10.50, 1.00)--(10.65, 2.00);
\draw[line width = 1.60pt, black!100.00!white] (10.35, 1.00)--(10.65, 2.00);
\draw[line width = 1.60pt, black!100.00!white] (10.20, 1.00)--(10.65, 2.00);
\draw[line width = 1.60pt, black!100.00!white] (9.30, 1.00)--(9.00, 2.00);
\draw[line width = 1.60pt, black!100.00!white] (8.85, 1.00)--(8.40, 2.00);
\draw[line width = 1.60pt, black!100.00!white] (8.70, 1.00)--(8.40, 2.00);
\draw[line width = 1.60pt, black!100.00!white] (8.55, 1.00)--(8.40, 2.00);
\draw[line width = 1.60pt, black!100.00!white] (8.25, 1.00)--(8.40, 2.00);
\draw[line width = 1.60pt, black!100.00!white] (8.10, 1.00)--(8.40, 2.00);
\draw[line width = 1.60pt, black!100.00!white] (7.65, 1.00)--(8.10, 2.00);
\draw[line width = 1.60pt, black!100.00!white] (7.50, 1.00)--(7.50, 2.00);
\draw[line width = 1.60pt, black!100.00!white] (11.25, 1.00)--(11.10, 2.00);
\draw[line width = 1.60pt, black!100.00!white] (15.15, 1.00)--(15.15, 2.00);
\draw[line width = 1.60pt, black!10.00!white] (13.95, 2.00)--(13.65, 3.00);
\draw[line width = 1.60pt, black!10.00!white] (13.95, 2.00)--(13.05, 3.00);
\draw[line width = 1.60pt, black!10.00!white] (12.90, 2.00)--(13.65, 3.00);
\draw[line width = 1.60pt, black!10.00!white] (15.15, 2.00)--(14.40, 3.00);
\draw[line width = 1.60pt, black!25.00!white] (9.00, 2.00)--(7.95, 3.00);
\draw[line width = 1.60pt, black!80.00!white] (9.00, 2.00)--(10.20, 3.00);
\draw[line width = 1.60pt, black!90.00!white] (13.95, 2.00)--(12.75, 3.00);
\draw[line width = 1.60pt, black!95.00!white] (15.15, 2.00)--(15.00, 3.00);
\draw[line width = 1.60pt, black!95.00!white] (12.90, 2.00)--(12.75, 3.00);
\draw[line width = 1.60pt, black!100.00!white] (15.00, 2.00)--(13.65, 3.00);
\draw[line width = 1.60pt, black!100.00!white] (13.80, 2.00)--(12.75, 3.00);
\draw[line width = 1.60pt, black!100.00!white] (12.30, 2.00)--(12.75, 3.00);
\draw[line width = 1.60pt, black!100.00!white] (12.00, 2.00)--(12.75, 3.00);
\draw[line width = 1.60pt, black!100.00!white] (11.85, 2.00)--(10.50, 3.00);
\draw[line width = 1.60pt, black!100.00!white] (6.90, 2.00)--(7.05, 3.00);
\draw[line width = 1.60pt, black!100.00!white] (10.65, 2.00)--(10.50, 3.00);
\draw[line width = 1.60pt, black!100.00!white] (10.35, 2.00)--(10.50, 3.00);
\draw[line width = 1.60pt, black!100.00!white] (10.05, 2.00)--(10.50, 3.00);
\draw[line width = 1.60pt, black!100.00!white] (9.60, 2.00)--(10.50, 3.00);
\draw[line width = 1.60pt, black!100.00!white] (9.45, 2.00)--(10.50, 3.00);
\draw[line width = 1.60pt, black!100.00!white] (9.30, 2.00)--(10.50, 3.00);
\draw[line width = 1.60pt, black!100.00!white] (8.40, 2.00)--(7.05, 3.00);
\draw[line width = 1.60pt, black!100.00!white] (8.10, 2.00)--(7.05, 3.00);
\draw[line width = 1.60pt, black!100.00!white] (7.50, 2.00)--(7.05, 3.00);
\draw[line width = 1.60pt, black!100.00!white] (11.70, 2.00)--(10.50, 3.00);
\draw[line width = 1.60pt, black!100.00!white] (11.10, 2.00)--(10.50, 3.00);
\draw[line width = 1.60pt, black!10.00!white] (15.00, 3.00)--(16.20, 4.00);
\draw[line width = 1.60pt, black!10.00!white] (13.05, 3.00)--(16.20, 4.00);
\draw[line width = 1.60pt, black!15.00!white] (7.95, 3.00)--(11.25, 4.00);
\draw[line width = 1.60pt, black!15.00!white] (10.20, 3.00)--(8.25, 4.00);
\draw[line width = 1.60pt, black!20.00!white] (12.75, 3.00)--(16.20, 4.00);
\draw[line width = 1.60pt, black!25.00!white] (13.05, 3.00)--(11.25, 4.00);
\draw[line width = 1.60pt, black!25.00!white] (10.20, 3.00)--(11.25, 4.00);
\draw[line width = 1.60pt, black!45.00!white] (13.65, 3.00)--(13.80, 4.00);
\draw[line width = 1.60pt, black!45.00!white] (14.40, 3.00)--(16.20, 4.00);
\draw[line width = 1.60pt, black!60.00!white] (13.65, 3.00)--(11.25, 4.00);
\draw[line width = 1.60pt, black!60.00!white] (14.40, 3.00)--(13.80, 4.00);
\draw[line width = 1.60pt, black!70.00!white] (10.20, 3.00)--(13.80, 4.00);
\draw[line width = 1.60pt, black!75.00!white] (13.05, 3.00)--(13.80, 4.00);
\draw[line width = 1.60pt, black!85.00!white] (12.75, 3.00)--(13.80, 4.00);
\draw[line width = 1.60pt, black!90.00!white] (7.95, 3.00)--(8.25, 4.00);
\draw[line width = 1.60pt, black!95.00!white] (15.00, 3.00)--(11.25, 4.00);
\draw[line width = 1.60pt, black!100.00!white] (10.50, 3.00)--(8.25, 4.00);
\draw[line width = 1.60pt, black!100.00!white] (7.05, 3.00)--(8.25, 4.00);
\draw[line width = 1.60pt, black!100.00!white] (8.25, 4.00)--(11.40, 5.00);
\draw[line width = 1.60pt, black!100.00!white] (11.25, 4.00)--(11.40, 5.00);
\draw[line width = 1.60pt, black!100.00!white] (13.80, 4.00)--(11.40, 5.00);
\draw[line width = 1.60pt, black!100.00!white] (16.20, 4.00)--(11.40, 5.00);
;
\fill (7.35, 1.00) circle (2.7pt);
\fill (7.50, 1.00) circle (1.7pt);
\fill (7.65, 1.00) circle (1.8pt);
\fill (8.10, 1.00) circle (1.4pt);
\fill (8.25, 1.00) circle (1.4pt);
\fill (8.55, 1.00) circle (0.4pt);
\fill (8.70, 1.00) circle (0.9pt);
\fill (8.85, 1.00) circle (2.7pt);
\fill (9.15, 1.00) circle (1.6pt);
\fill (9.30, 1.00) circle (1.1pt);
\fill (9.75, 1.00) circle (0.8pt);
\fill (9.90, 1.00) circle (1.7pt);
\fill (10.20, 1.00) circle (2.0pt);
\fill (10.35, 1.00) circle (2.4pt);
\fill (10.50, 1.00) circle (0.7pt);
\fill (10.80, 1.00) circle (0.1pt);
\fill (10.95, 1.00) circle (1.8pt);
\fill (11.10, 1.00) circle (1.2pt);
\fill (11.25, 1.00) circle (0.5pt);
\fill (11.40, 1.00) circle (0.1pt);
\fill (11.55, 1.00) circle (0.3pt);
\fill (11.70, 1.00) circle (3.2pt);
\fill (12.00, 1.00) circle (2.7pt);
\fill (12.15, 1.00) circle (0.8pt);
\fill (12.30, 1.00) circle (0.2pt);
\fill (12.45, 1.00) circle (0.6pt);
\fill (12.60, 1.00) circle (0.0pt);
\fill (12.75, 1.00) circle (0.3pt);
\fill (12.90, 1.00) circle (0.1pt);
\fill (13.05, 1.00) circle (0.3pt);
\fill (13.20, 1.00) circle (1.4pt);
\fill (13.65, 1.00) circle (2.3pt);
\fill (14.10, 1.00) circle (1.4pt);
\fill (14.55, 1.00) circle (0.8pt);
\fill (15.00, 1.00) circle (0.6pt);
\fill (15.15, 1.00) circle (1.2pt);
\fill (6.90, 2.00) circle (0.6pt);
\fill (7.50, 2.00) circle (1.1pt);
\fill (8.10, 2.00) circle (2.4pt);
\fill (8.40, 2.00) circle (3.1pt);
\fill (9.00, 2.00) circle (1.7pt);
\fill (9.30, 2.00) circle (1.3pt);
\fill (9.45, 2.00) circle (1.1pt);
\fill (9.60, 2.00) circle (0.7pt);
\fill (10.05, 2.00) circle (1.2pt);
\fill (10.35, 2.00) circle (1.4pt);
\fill (10.65, 2.00) circle (3.6pt);
\fill (11.10, 2.00) circle (1.5pt);
\fill (11.70, 2.00) circle (1.2pt);
\fill (11.85, 2.00) circle (1.1pt);
\fill (12.00, 2.00) circle (0.6pt);
\fill (12.30, 2.00) circle (2.7pt);
\fill (12.90, 2.00) circle (2.0pt);
\fill (13.80, 2.00) circle (2.1pt);
\fill (13.95, 2.00) circle (2.2pt);
\fill (15.00, 2.00) circle (1.9pt);
\fill (15.15, 2.00) circle (3.2pt);
\fill (7.05, 3.00) circle (3.1pt);
\fill (7.95, 3.00) circle (2.0pt);
\fill (10.20, 3.00) circle (2.0pt);
\fill (10.50, 3.00) circle (7.4pt);
\fill (12.75, 3.00) circle (3.3pt);
\fill (13.05, 3.00) circle (2.8pt);
\fill (13.65, 3.00) circle (2.3pt);
\fill (14.40, 3.00) circle (1.7pt);
\fill (15.00, 3.00) circle (2.1pt);
\fill (8.25, 4.00) circle (4.1pt);
\fill (11.25, 4.00) circle (3.8pt);
\fill (13.80, 4.00) circle (4.0pt);
\fill (16.20, 4.00) circle (3.8pt);
\fill (11.40, 5.00) circle (4.2pt);
\end{tikzpicture}

%% file: proofs.tex

\section{Proofs}
First, in Section \ref{lowSec}, we establish the lower bound of Theorem \ref{thm1}. The main step is to show that the relevant edges $\EE$ form a forest. That is, with probability 1, every node has precisely one outgoing edge in $\EE$. The argument critically relies on the assumption of strong reinforcement, where $\a > 1$. 

Next, the additive upper bound in Theorem \ref{thm2} is stronger than the multiplicative upper bound in Theorem \ref{thm1}, so that it suffices to establish the former. To achieve this goal, in Section \ref{up1Sec}, we first give a short and instructive proof for $\aa \ge 3$.  The heavy-tailedness of the fitness distribution ensures that although the number of possible connections from each node grows exponentially in the layer, with positive probability, $\EE$ contains the edge leading to the node with maximal fitness in the next layer. Then, in Section \ref{up2Sec}, we work out the more subtle arguments for general $\aa > 1$.

\subsection{Lower bound}
\label{lowSec}
The main step in the lower bound is to prove that $\EE$ is a forest. Essentially, this follows from a variant of the celebrated Rubin's theorem for  P\'olya urns in the regime of strong reinforcement.

%
%
\begin{lemma}[$\EE$ is a forest]
	\label{forLem}
	With probability 1, $\EE$ is a forest.
\end{lemma}
\begin{proof}
	The critical observation is that the outgoing edges adjacent to a node $v$ are only reinforced at Poisson clock rings at the vertex $v$. Hence, we may view these edges as colors in a P\'olya urn governed by a super-linear reinforcement scheme. Then, by the celebrated Rubin's theorem, almost surely all but one of the edges are reinforced only finitely often. For two colors, this is shown in \cite[Theorem 3.6]{pemantle}, and a generalization to an arbitrary finite number can be found in \cite[Theorem 3.3.1]{zhu}.
\end{proof}

%
%
Hence, in each layer $h \ge 0$, there exists almost surely a unique node $(L_h, h)$ such that $(0, 0)$ connects to $(L_h, h)$ by a directed path in $\EE$. Similarly, we write $(R_h, h)$ for the unique node in layer $h$ connected along a directed path to $(N, 0)$. In this notation, $H_N$ is twice the coalescence time of $L_h$ and $R_h$, i.e.,
\begin{align}
	\label{coalRepEq}
	H_N = 2 \inf\{h \ge 1:\, L_h = R_h\}.
\end{align}

%
%
Now, the lower bound becomes a consequence of the structure of the underlying deterministic graph $(V, E)$. More precisely, we first establish an auxiliary result on the growth of the difference $D_h = R_h - L_h$. 
\begin{lemma}[Growth of $D_h$]
	\label{distLem}
	Let $h, h' \ge 0$. Then,
$$		|D_{h + h'} - D_h| \le 2 \sum_{h \le i < h + h' } \aa^i \le \frac{2 \aa^{h + h'}}{\aa - 1}.$$
\end{lemma}
\begin{proof}
By definition, any node in layer $i$ can connect to nodes in layer $i + 1$ at horizontal distance at most $\aa^i$, so that for every $i \ge 0$,
	$$\max\{|L_i - L_{i + 1}|, |R_i - R_{i + 1}|\} \le \aa^i.$$
	In particular, 
	$$|D_{h + h'} - D_h| \le \sum_{h \le i < h + h' } |D_{i + 1} - D_i| \le 2 \sum_{h \le i < h + h'} \aa^i$$
	The second inequality in the assertion follows from the geometric series representation.
\end{proof}

Now, we have all ingredients to prove the lower bound in Theorem \ref{thm1}.
\begin{proof}[Proof of Theorem \ref{thm1}, lower bound]
	Since $D_{H_N / 2 } = 0$ and $D_0 = N$, an application of Lemma \ref{distLem} gives that
	$$ \log_\aa(2 / (a - 1)) + H_N / 2 \ge \log_\aa(N),$$
	as asserted. 
\end{proof}

\subsection{Theorem \ref{thm2}; $\aa \ge 3$} 
\label{up1Sec}

Lemma \ref{forLem} produces for every node a unique outgoing edge that is reinforced infinitely often. Leveraging the heavy-tailedness of the fitness distribution, a key ingredient in the proof of the upper bound is that with a probability bounded away from 0, this edge leads to the node with maximal fitness. To make this precise, let $\fmv \in \NN_v$ be the out-neighbor of $v \in V$ with maximal fitness. We let 
$$\FF_h = \sigma(\{L_{h'}, R_{h'}\}_{h' \le h}, \{F_v\}_{v \in \Z \times \{0, \dots, h\}})$$
denote the $\sigma$-algebra of the information gathered up to layer $h$ and write
$$\FF_h^* = \sigma(\{L_{h'}, R_{h'}\}_{h' \le h}, \{F_v\}_{v \in \Z \times \{0, \dots, h + 1\}})$$
for the $\sigma$-algebra that additionally contains the information on the fitnesses in layer $h + 1$.

%
%
\begin{lemma}[Choice of the fittest]
	\label{fitLem}
	For every $\e > 0$ there exists $q_\e > 0$ such that almost surely
	$$ \inf_{\substack{h \ge 0 \\ v \in \Z \times \{h\} }}\P(\{v, \fmv\} \in \EE \,|\, \FF_h^*) \ge q_\e \one\Big\{\max_{w \in \NN_v}F_w \ge \e \sum_{w \in \NN_v}F_w\Big\}.$$
\end{lemma}
\begin{proof}
	%
	%
	Let $\t_n$ denote the $n$th firing time at the node $v$ and write
	$$E_v^n = \big\{W_{\t_n}(\{v, \fmv\}) = W_{\t_n-}(\{v, \fmv\}) + 1\big\}$$ 
	for the event that at time $\t_n$ the edge $\{v, \fmv\}$ is reinforced. In particular, $\{\{v, \fmv\} \in \EE\} \supset \cap_{n \ge 1} E_v^n$ and  
	\begin{align*}
		\P\Big(E_v^n\,\Big|\FF_h^*, \bigcap_{k \le n - 1} E_v^k\Big) &\ge \frac{F_{\fmv} n^\a}{F_{\fmv} n^\a + \sum_{\substack{w \in \NN_v \setminus\{\fmv\}}}F_w}.
	\end{align*}
	Therefore, putting $M_v =  \max_{w \in \NN_v}F_w$ and $S_v  = \sum_{w \in \NN_v}F_w$, we obtain that almost surely, 
	\begin{align*}
		\P(\{v, \fmv\} \in \EE \,|\,\FF_h^*) &\ge  \prod_{n \ge 1}\frac{M_v n^\a}{M_v n^\a + S_v}.
	\end{align*}
	In particular, 
	\begin{align*}
		\P(\{v, \fmv\} \in \EE \,|\,\FF_h^*) &\ge \one\{M_v \ge \e S_v\} \prod_{n \ge 1}(1 - (1 +\e  n^\a)^{-1}).
	\end{align*}
Since the series $\sum_{n \ge 1}(1 + \e n^\a)^{-1}$ converges, the product $q_\e = \prod_{n \ge 1}(1 - (1 +\e  n^\a)^{-1})$ is strictly positive, as asserted.
\end{proof}

To show that the sum and the maximum of the fitnesses appearing in Lemma \ref{fitLem} are of the same order, we critically rely on the assumption that the fitnesses are heavy-tailed. In order to be applicable both for the case $\aa \ge 3$ as well as for $\aa < 3$, we provide a slightly more refined result, where we compare the second largest value of iid heavy-tailed Pareto random variables to the sum. For this purpose, we write $\max^{(2)}_{i \le m}x_i$ for the second largest value among real numbers $x_1,\dots, x_m$.

%
%
\begin{lemma}[Sum vs.~Second-largest value for Pareto random variables]
	\label{paretoLem}
	Let $\{X_i\}_{i \ge 1}$ be iid Pareto random variables with parameter $\g < 1$. Let $S_m = \sum_{i \le m} X_i$ and let $M_m^{(2)} = \max_{i \le m}^{(2)} X_i$. Then, $\{S_m / M_m^{(2)}\}_{m \ge 1}$ is tight.
\end{lemma}
Mind that Lemma \ref{paretoLem} implies readily that $\{S_m / \max_{i \le m}X_i\}_{m \ge 1}$ is tight as well.
\begin{proof}
	First, by the stable limit theorem, the scaled sum $m^{-1/\g}S_m$ converges to a stable distribution \cite[Theorem XVII.5.3]{feller}. Second, by extremal value theory, the scaled maximum $m^{-1/\g}\max_{i \le m}X_i$ converges to a Fr\'echet distribution, whereas the ratio $\max_{i \le m}X_i / M_m^{(2)}$ converges to 1 \cite[Theorem 3.3.7, Example 4.1.11]{extreme}. This yields tightness of $S_m / M_m^{(2)}$.
\end{proof}

Now, we prove Theorem \ref{thm2} for $\aa \ge 3$. The key simplification in the case $\aa \ge 3$ is that for every $h \ge 0$, the set of possible coalescence nodes grows so quickly that the conditional probability of coalescence in step $h + 1$ given the information in $\FF_h$ is bounded away from $0$ uniformly in $h \ge 0$.

\begin{proof}[Proof of Theorem \ref{thm2}, $\aa \ge 3$]
	%
	%
	To prove the result, we first assert that there exists $\delta > 0$ such that
	$$\P(L_{h + 1} = R_{h + 1}\,|\, \FF_h) \ge  \delta$$
holds for every $h \ge \log_a(N) + \log_a(2)$.
	Once this assertion is shown, we obtain that for  $x \ge  2 \log_a(2)$, 
	\begin{align*}
		\P(H_N \ge 2 \log_a(N) + x) &\le \P(L_h \ne R_h \text{ for all $h \le \log_a(N) + x/2$}) \\
		&\le (1 - \delta )^{\lfloor x/2 -  \log_a(2)\rfloor},
	\end{align*}
	which decays exponentially fast in $x$.

	To prove the asserted lower bound,  we first introduce 
	\begin{align}
		\label{cmEq}
		C_h^+ = \NN_{(L_h, h)} \cup  \NN_{(R_h, h)} \quad \text{ and }\quad C_h^- = \NN_{(L_h, h)} \cap  \NN_{(R_h, h)}
	\end{align}
	as the union and intersection of the out-neighborhoods of ${(L_h, h)}$ and ${(R_h, h)}$, respectively.
		Then, 
	$$\{L_{h + 1} = R_{h + 1}\} \supset \{L_h^{\mx} = R_h^{\mx}\} \cap A_h,$$
	where 
	$$A_h = \{L_{h + 1} = L_h^{\mx}\} \cap \{R_{h + 1} = R_h^{\mx}\}.$$
	Therefore, by Lemma \ref{fitLem}, for every $\e > 0$,
	\begin{align*}
		&\P(L_{h + 1} = R_{h + 1}\,|\, \FF_h) \\
		&\quad \ge \P\Big(\one\{L_h^{\mx} = R_h^{\mx}\}\P(A_h\,|\, \FF_h^*)\,\Big|\, \FF_h\Big)\\
		&\quad \ge q_\e^2 \P\Big(\one\{L_h^{\mx} = R_h^{\mx}\}\cap \Big\{\max_{w \in C_h^+}F_w \ge \e \sum_{w \in C_h^+}F_w\Big\}\,\Big|\, \FF_h\Big).
	\end{align*}
	Since the positions $L_h^{\mx}$, $R_h^{\mx}$ of the maximal fitnesses are independent of the value of the sum and the value of the maximum of the relevant fitnesses, we arrive at 
	\begin{align*}
		\P(L_{h + 1} = R_{h + 1}\,|\, \FF_h) 
		 \ge q_\e^2 \P(L_h^{\mx} = R_h^{\mx}\,|\, \FF_h)\P\Big(\Big\{\max_{w \in C_h^+}F_w \ge \e \sum_{w \in C_h^+}F_w\Big\}\,\Big|\, \FF_h\Big).
	\end{align*}
		By Lemma \ref{paretoLem}, the second probability is bounded below by $1/2$ for sufficiently small $\e > 0$.  Hence, it remains to provide a lower bound for $\P(L_h^{\mx} = R_h^{\mx}\,|\, \FF_h)$.

	 We write $P_h \in \Z \times \{h + 1\}$ for the position of the maximal fitness in $C_h^+$, i.e., 
	$$F_{P_h} = \max_{w \in C_h^+}F_w.$$
Then, $P_h$ is uniformly distributed in $C_h^+$, so that 
	$$\P(L_h^{\mx} = R_h^{\mx}\,|\, \FF_h) = \P(P_h \in C_h^-\,|\, \FF_h) = \frac{\#C_h^-}{\#C_h^+} \ge \frac{\#C_h^-}{4 \lfloor a^h\rfloor + 2}.$$
	Finally, to derive a lower bound on $\#C_h^-$, Lemma \ref{distLem} gives that, for every $h \ge \log_a(N) + \log_a(2)$,
	$$|L_h - R_h| \le  N + \frac{2a^h}{a - 1} \le N +  a^h \le \frac32 a^h.$$
	Therefore, 
	$$\#C_h^- \ge 2\lfloor a^h\rfloor - \frac32 a^h \ge \frac14 a^h,$$ 
	which implies the required lower bound.
\end{proof}

\subsection{Theorem \ref{thm2}; $\aa < 3$} 
\label{up2Sec}
After having developed the intuition behind the proof of Theorem \ref{thm2} for $\aa \ge 3$, we now assume that $\aa < 3$. The arguments in this case are more involved since it may happen that $L_h$ and $R_h$ are so far away that the set $C_h^-$ of possible coalescence points from \eqref{cmEq} becomes empty.  We deal with this problem by imposing that $L_h$ and $R_h$ both do not move substantially for a finite number of steps, which guarantees that the set of possible coalescence points becomes non-empty again.

We start by showing that coalescence occurs with positive probability after a small number of steps if initially $L_h$ and $R_h$ are not too far apart.
\begin{lemma}[$H_N$ is small with positive probability]
	\label{boundPosLem}
There exists $k \ge 1$ such that
	\begin{align*}
		\inf_{\substack{h, N\ge 0 \\ z:\, |z| \le 4\aa^h / (\aa - 1)}} \P(L_{h'} = R_{h'}\text{ for some $ h' \in  [h,   h + k - 1]$}\, |\, L_h - R_h = z) > 0.
	\end{align*}
\end{lemma}
\begin{proof}
	%
	%
	First, if $|z| \le \aa^h$, then $\#C_h^- \ge \aa^h / 2$ for large $h \ge 0$, so that arguing as in Section \ref{up1Sec} yields that
	$$\P(L_{h + 1} = R_{h + 1} \,|\, L_h - R_h = z) \ge \delta_0$$
	for a sufficiently small value of $\delta_0 > 0$.

	%
	%
	Hence, we may assume that $|z| > \aa^h$ and introduce the events 
	\begin{align}
		\label{ehDef}
		E_{h'} =\{L_{h' + 1} = R_{h' + 1}\} \cup \{\max\{|L_{h' + 1} - L_{h'}|, |R_{h' + 1} - R_{h'}|  \} \le \e_1 a^{h'}\},
	\end{align}
	where $\e_1 = (\aa - 1) / 8$. We assert that there exists $\delta > 0$ such that for every $h' \ge h$
		%
	%
	\begin{align}
		\label{fLowEq}
		\P(E_{h'}\,|\, \FF_{h'}) \ge \delta.
	\end{align}
	Before proving \eqref{fLowEq}, we show how to conclude the proof of the lemma.	First, set 
	$$h_1 = \min\big\{h' \ge h:\, |z| + 2\e_1 \sum_{h \le i \le h' - 1} \aa^i \le \frac12{\aa^{h' }}\,\big\}.$$
	In particular,
	\begin{align*}
		|z| &\ge \frac{\aa^{h_1 - 1}}2 - 2\e_1 \sum_{h \le i \le h_1 - 2} \aa^i	= \frac{\aa^{h_1 - 1}}2 - \frac{2\e_1 (\aa^{h_1 - 1} - \aa^h)}{\aa - 1} 		\ge \frac{\aa^{h_1 - 1}}4.
	\end{align*}
	Then, $|z| \le 4\aa^h / (\aa - 1)$ implies that 
		$h_1 - h - 1 \le \log_{\aa}(16 / (\aa - 1)).$
	Note that if $E_{h'} \cap \{L_{h'} \ne R_{h'}\}$ occurs for every $h' \in [h,  \dots, h_1 - 1]$, then 
	\begin{align*}
		|D_{h_1}| \le \Big||D_{h_1}| - |z|\Big| +  |z| \le \frac{2\e_1\aa^{h_1}}{\aa - 1}  + \frac{\aa^{h_1}}2 \le {\aa^{h_1}},
	\end{align*}
	where in the second inequality, we insert the definition of $h_1$ to bound $|z|$. Hence, by the case considered at the beginning of the proof, the conditional probability that $L_{h_1 + 1} = R_{h_1 + 1}$ given that $\cap_{h \le h' < h_1} E_{h'}$ is bounded below by $\delta_0$. Taking everything together, we arrive at the asserted positive lower bound 
	$$\P\Big(\{L_{h_1 + 1} = R_{h_1 + 1}\} \cap \bigcap_{h \le h' < h_1} E_{h'}\,\Big|\, L_h - R_h = z \Big) \ge \delta_0 \delta^{h_1 - h} \ge \delta_0 \delta^{\log_a(16 / (\aa - 1))}.$$

	It remains to establish \eqref{fLowEq}. To that end, we first set as before
	$$A_{h'} =  \{L_{h' + 1} = L_{h'}^{\mx}\} \cap \{R_{h' + 1} = R_{h'}^{\mx}\}.$$
	Then, we let $P_{h'} = (p_{h'}, h' + 1)$ and $P^{(2)}_{h'} = (p_{h'}^{(2)}, h' + 1)$ denote the positions of the largest and the second largest fitness in the union set $C_{h'}^+$ as defined in \eqref{cmEq}. That is, 
	$$F_{P_{h'}} = \max_{w \in C_{h'}^+}F_w\quad\text{ and }\quad F_{P_{h'}^{(2)}} = \max_{w \in C_{h'}^+}{}^{\hspace{-.1cm}(2)}F_{w}.$$
	Now,  define  
	$$E'_{h'} =\big\{\max\{|p_{h'} - L_{h'}|, |p_{h'}^{(2)} - R_{h'}| \} \le \e_1 a^{h'}\big\}$$
	as the event that the distances between $p_{h'}$ and $L_{h'}$, as well as between $p_{h'}^{(2)}$ and $R_{h'}$ are at most $\e_1 a^{h'}$.
	Then, we claim that
	$$E_{h'} \supset E_{h'}' \cap A_{h'}.$$
	Indeed, assume that $E_{h'}'$ occurs. In that case, if $P_{h'}$ is contained in the intersection set $C_{h'}^-$, then $L_{h'}^{\mx} = R_{h'}^{\mx}$. Otherwise, $p_{h'} = L_{h'}^{\mx}$ and $p_{h'}^{(2)} = R_{h'}^{\mx}$, so that $\max\{|L_{h'}^{\mx} - L_{h'}|, |R_{h'}^{\mx} - R_{h'}| \} \le \e_1 a^{h'}$. In particular,
	$$\{L_{h'}^{\mx} = R_{h' }^{\mx}\} \cup \{\max\{|L_{h' }^{\mx} - L_{h'}|, |R_{h' }^{\mx} - R_{h'}|  \} \le \e_1 a^{h'}\}$$ 
	occurs.
	Hence, under $A_{h'}$ the previous line becomes the defining equation for $E_{h'}$ as in \eqref{ehDef}.

	Now, arguing as in the case $a \ge 3$, we derive that for every $\e > 0$,
	\begin{align*}
		\P(E_{h'}\,|\, \FF_h) 
		\ge q_\e^2 \P(E'_{h'}\,|\, \FF_h)\P\Big(\Big\{\max_{w \in C_h^+}{}^{\hspace{-.1cm}(2)}F_w \ge \e \sum_{w \in C_h^+}F_w\Big\}\,\Big|\, \FF_h\Big).
	\end{align*}
	Note that here, we need to consider the second largest value in $C_h^+$ since under $E'_{h'}$ the positions $L_{h'}^{\mx}$ and $R_{h'}^{\mx}$ could be distinct. By Lemma \ref{paretoLem}, it therefore suffices to derive a lower bound on $\P(E'_{h'}\,|\, \FF_h)$.

	Finally, since $p_h$ and $p_{h'}^{(2)}$ are uniform in $C_{h'}^+$, we obtain that for large $h \ge 0$
	$$\P\Big(\max\{|p_{h'} - L_{h'}|, |p_{h'}^{(2)} - R_{h'}| \} \le \e_1 a^{h'}\,|\, \FF_{h'}\Big) \ge \Big(\frac{\e_1 a^{h'}}{4a^{h'} + 2}\Big)^2,$$
	which is bounded away from 0, thereby completing the proof of \eqref{fLowEq}.
\end{proof}

\begin{proof}[Proof of Theorem \ref{thm2}, $a < 3$]
	First, using $1 \le 2 / (a - 1) $ for $a < 3$, Lemma \ref{distLem} implies that for every $h \ge h_0 = \log_a(N)$,
	\begin{align*}
		|D_h| \le N + \frac{2 \aa^h}{\aa - 1} \le  \frac{4 \aa^h}{\aa - 1},
	\end{align*}
so that we are in a position to apply Lemma \ref{boundPosLem}. Choosing $k \ge 1$ as in that lemma, we let
	$$G_i = \{R_h \ne L_h \text{ for every $h \in [ik, i(k + 1) - 1]$}\}$$
	denote the event that we do not see coalescence in the interval $[ik, i(k + 1) - 1]$ and put $G_i' = \cap_{i' \le i}G_{i'}$.
	 In particular, under the event $\{H_N \ge 2 \log_a(N) + x\}$, the event $G'_{i_1}$ occurs 	for $i_1  = \lfloor (\log_a(N) + x/2) / k\rfloor$. Hence, by the Markov property at time $(i_1 - 1)$ and Lemma \ref{boundPosLem}, we have a constant $\e > 0$ such that
	\begin{align*}
		\P( G_{i_1}')
		&\le \E\Big[\P\big(L_h \ne R_h\text{ for every $h \in [(i_1 - 1)k, i_1k - 1]$}\,\big|\,D_{(i_1 - 1)k}\big) \one\{G_{i_1 - 1}'\}\Big]\\
		&\le (1 - \e)\P( G_{i_1 - 1}').
	\end{align*}
	Hence, putting $i_0 = \lceil \log_a(N) / k\rceil$, we conclude that 
	$$\P(H_N \ge 2 \log_a(N) + x) \le (1 - \e)^{i_1 - i_0},$$
	which decays exponentially fast in $x$.  
\end{proof}

%% file: main.bbl
\begin{thebibliography}{10}

\bibitem{BarabasiAlbert99}
A.-L. {Barab{\'a}si} and R.~{Albert}.
\newblock {Emergence of scaling in random networks}.
\newblock {\em Science}, 286:509--512, 1999.

\bibitem{BengioLBL15}
Y.~Bengio, D.~Lee, J.~Bornschein, and Z.~Lin.
\newblock Towards biologically plausible deep learning.
\newblock {\em CoRR}, abs/1502.04156, 2015.

\bibitem{BollobasRiordan04}
B.~{Bollob\'as} and O.~{Riordan}.
\newblock {The diameter of a scale-free random graph.}
\newblock {\em {Combinatorica}}, 24(1):5--34, 2004.

\bibitem{BonacichLiggett03}
P.~{Bonacich} and T.~M. {Liggett}.
\newblock {Asymptotics of a matrix valued Markov chain arising in sociology.}
\newblock {\em {Stochastic Processes Appl.}}, 104(1):155--171, 2003.

\bibitem{warm4}
Y.~Couzini\'e and C.~Hirsch.
\newblock Infinite {WARM} graphs {I}. {W}eak reinforcement regime.
\newblock In preparation.

\bibitem{DelattreFournierHoffmann16}
S.~{Delattre}, N.~{Fournier}, and M.~{Hoffmann}.
\newblock {Hawkes processes on large networks.}
\newblock {\em {Ann. Appl. Probab.}}, 26(1):216--261, 2016.

\bibitem{extreme}
P.~Embrechts, C.~Kl\"{u}ppelberg, and T.~Mikosch.
\newblock {\em Modelling Extremal Events}.
\newblock Springer, Berlin, 1997.

\bibitem{feller}
W.~Feller.
\newblock {\em An Introduction to Probability Theory and its Applications.
  {V}ol. {II}}.
\newblock Second edition. J. Wiley \& Sons, New York, 1971.

\bibitem{warm3}
C.~Hirsch, M.~Holmes, and V.~Kleptsyn.
\newblock Absence of {WARM} percolation in the very strong reinforcement
  regime.
\newblock Preprint available at
  https://christian-hirsch.github.io/publications.html.

\bibitem{warm1}
R.~v.~d. Hofstad, M.~Holmes, A.~Kuznetsov, and W.~Ruszel.
\newblock Strongly reinforced {P}{\'o}lya urns with graph-based competition.
\newblock {\em Ann. Appl. Probab.}, 26(4):2494--2539, 2016.

\bibitem{warm5}
M.~Holmes and V.~Kleptsyn.
\newblock Infinite {WARM} graphs {II}. {C}ritical regime.
\newblock In preparation.

\bibitem{warm2}
M.~Holmes and V.~Kleptsyn.
\newblock Proof of the {WARM} whisker conjecture for neuronal connections.
\newblock {\em Chaos}, 27(4):043104, 2017.

\bibitem{JacobMoerters13}
E.~{Jacob} and P.~{M\"orters}.
\newblock {A spatial preferential attachment model with local clustering.}
\newblock In A.~Bonato, M.~Mitzenmacher, and P.~Pra{\l}at, editors, {\em
  Proceedings of the 10th {I}nternational {W}orkshop ({WAW} 2013) held at
  {H}arvard {U}niversity, {C}ambridge, {MA}, {D}ecember 14--15, 2013}, pages
  14--25. Berlin: Springer, 2013.

\bibitem{kalisman-silberberg-markram-2005}
N.~Kalisman, G.~Silberberg, and H.~Markram.
\newblock The neocortical microcircuit as a tabula rasa.
\newblock {\em Proceedings of the National Academy of Sciences},
  102(3):880--885, 2005.

\bibitem{LiggettRolles04}
T.~M. {Liggett} and S.~W.~W. {Rolles}.
\newblock {An infinite stochastic model of social network formation.}
\newblock {\em {Stochastic Processes Appl.}}, 113(1):65--80, 2004.

\bibitem{pemantle}
R.~Pemantle.
\newblock A survey of random processes with reinforcement.
\newblock {\em Probab. Surv.}, 4:1--79, 2007.

\bibitem{SkyrmsPemantle04}
R.~{Pemantle} and B.~{Skyrms}.
\newblock {Network formation by reinforcement learning: the long and medium
  run.}
\newblock {\em {Math. Soc. Sci.}}, 48(3):315--327, 2004.

\bibitem{SkyrmsPemantle00}
B.~{Skyrms} and R.~{Pemantle}.
\newblock {A dynamic model of social network formation.}
\newblock {\em {Proc. Natl. Acad. Sci. USA}}, 97(16):9340--9346, 2000.

\bibitem{zhu}
T.~Zhu.
\newblock {\em Nonlinear {P}{\'o}lya urn models and self-organizing processes}.
\newblock PhD thesis, University of Pennsylvania, 2009.

\end{thebibliography}
